\documentclass[oneside, a4paper, 11pt]{article}
\usepackage[colorlinks = true,
            linkcolor = black,
            urlcolor  = blue,
            citecolor = black]{hyperref}
\usepackage{amsmath, amsthm, amssymb, amsfonts, amscd}
\usepackage[left=2.5cm,right=2.5cm,top=3cm,bottom=3cm,a4paper]{geometry}
\usepackage{mathtools}
\usepackage{thmtools}
\usepackage{scalerel}
\usepackage{graphicx}
\usepackage[titletoc]{appendix}
\usepackage{cleveref}
\usepackage{mathrsfs}
\DeclareGraphicsExtensions{.pdf,.png,.jpg}

\theoremstyle{definition}

\newcommand{\Gal}{\operatorname{Gal}}

\newcommand{\End}{\operatorname{End}}

\newcommand{\Z}{\mathbb{Z}}
\newcommand{\Q}{\mathbb{Q}}

\newcommand{\F}{\mathbb{F}}

\newtheorem{theorem}{Theorem}[section]
\newtheorem{proposition}[theorem]{Proposition}

\newtheorem{lemma}[theorem]{Lemma}
\newtheorem{conjecture}[theorem]{Conjecture}
\newtheorem*{conjecture*}{Conjecture}

\usepackage{titling}
\usepackage{url}
\usepackage[nosort]{cite}
\newcommand{\changeurlcolor}[1]{\hypersetup{urlcolor=#1}} 

\usepackage{enumitem}
\usepackage{tikz-cd}

\usepackage{array}
\makeatletter
\newcommand{\thickhline}{%
    \noalign {\ifnum 0=`}\fi \hrule height 1pt
    \futurelet \reserved@a \@xhline
}
\newcolumntype{"}{@{\hskip\tabcolsep\vrule width 1pt\hskip\tabcolsep}}
\makeatother

\title{\large{\textbf{ON A NUMBER OF ISOGENY CLASSES OF SIMPLE ABELIAN VARIETIES OVER FINITE FIELDS}}}
\author{\normalsize{JUNGIN LEE}}
\date{}

\usepackage{fancyhdr}

\newcommand\shorttitle{ON A NUMBER OF ISOGENY CLASSES OF SIMPLE ABELIAN VARIETIES OVER FINITE FIELDS}
\newcommand\authors{JUNGIN LEE}

\usepackage{sectsty}
\sectionfont{\large}
\subsectionfont{\normalsize}

\begin{document}
\maketitle

\vspace{-10mm}

\begin{abstract}
In this paper, we investigate the asymptotic behavior of the number $s_q(g)$ of isogeny classes of simple abelian varieties of dimension $g$ over a finite field $\F_q$. 
We prove that the logarithmic asymptotic of $s_q(g)$ is the same as the logarithmic asymptotic of the number $m_q(g)$ of isogeny classes of all abelian varieties of dimension $g$ over $\F_q$. 
We also prove that
$$
\limsup_{g \rightarrow \infty}\frac{s_q(g)}{m_q(g)}=1.
$$
This suggests that there are much more simple isogeny classes of abelian varieties over $\F_q$ of dimension $g$ than non-simple ones for sufficiently large $g$, which can be understood as the opposite situation to a main result of Lipnowski and Tsimerman (Duke Math 167:3403-3453, 2018). 
\end{abstract}


\section{Introduction} \label{Sec1}

An abelian variety over an arbitrary field $k$ is isogenous to a product of simple abelian varieties. 
One can naturally ask about the distribution of the dimensions of simple isogeny factors in the set of isogeny classes of abelian varieties over $k$ of given dimension. 
When $k$ is a prime field $\F_p$ and a variety is equipped with a principal polarization, there is a nice answer given by M. Lipnowski and J. Tsimerman \cite{LT}. 
Note that we may replace $0.99$ in the proposition below by any constant $c<1$. 

\begin{proposition} \label{prop11}
(\cite{LT}, Corollary 5.14) Under the assumption of \cite[Conjecture 5.2]{LT}, for a subset of primes $p$ of density at least $1-2^{-9}$, the proportion of principally polarized abelian varieties over $\F_p$ which admits an isogeny factor $E^h$ for some elliptic curve $E$ and $h \geq 0.99g$ approaches $1$ as $g \rightarrow \infty$.
\end{proposition}

The purpose of this article is to answer the question for abelian varieties (without polarization) over a finite field $\F_q$. An interesting point is that if we do not consider polarizations, most of the isogeny classes have a large simple isogeny factor, which is opposite to the case with principal polarizations. In particular, the number of isogeny classes of simple abelian varieties over $\F_q$ is large. We review some background material and summarize the results of this paper in the rest of this section. 

Let $A$ be an abelian variety over a finite field $\F_q$ of dimension $g$.
For a prime $\ell \nmid q$, there is a bijection
$$
\Q_{\ell} \otimes_{\Z} \End_{\F_q}(A) \cong \End_{\Q_{\ell}}(V_{\ell}(A))
$$
due to Tate \cite{TATE} so the $q$-Frobenius endomorphism on $A$ corresponds to an endomorphism on a $\Q_{\ell}$-vector space $V_{\ell}(A)$. Denote its characteristic polynomial by $p_A$. (It is called the Weil $q$-polynomial.)

Then $p_A$ is independent of the choice of $\ell$, monic, has integer coefficients, of degree $2g$ and all of its roots are Weil $q$-numbers (i.e. algebraic integers all of whose $\Gal(\overline{\Q}/\Q)$-conjugates have an absolute value $\sqrt{q}$) by the Riemann hypothesis (Weil conjecture). 
By Honda-Tate theorem \cite{HONDA, TATE}, two abelian varieties $A$ and $B$ over $\F_q$ are isogenous if and only if $p_A=p_B$. This enables us to count the number of isogeny classes of abelian varieties over $\F_q$ of given dimension by observing the Weil $q$-polynomials.

Let $m_q(g)$ be the number of isogeny classes of abelian varieties over $\F_q$ of dimension $g$. 
Then by \cite[Corollary 2.3]{LT} and \cite[Lemma 3.3.1]{DH}, the asymptotic formula for $m_q(g)$ as $g \rightarrow \infty$ is given by
\begin{equation} \label{eqn1}
m_q(g) = q^{\frac{1}{4}g^2(1+o(1))}.
\end{equation}
In Section \ref{Sec2}, we prove that the number of isogeny classes of simple abelian varieties over $\F_q$ of dimension $g$ have the same magnitude. 

\begin{theorem} \label{thm12}
(Theorem \ref{thm23}) Let $s_q(g)$ be the number of isogeny classes of simple abelian varieties over $\F_q$ of dimension $g$. Then
\begin{equation} \label{eqn2}
s_q(g) = q^{\frac{1}{4}g^2(1+o(1))}.
\end{equation}
\end{theorem}

Section \ref{Sec3} is devoted to the distribution of simple isogeny factors of the isogeny classes of abelian varieties over $\F_q$. 
Contrary to the case with principal polarizations (given in Proposition \ref{prop11}), most of the isogeny classes have a large simple isogeny factor. 
Precisely, we prove that for any $\varepsilon>0$, 
\begin{equation} \label{eqn3}
\lim_{g \rightarrow \infty} \frac{a_{\text{L}}(q, g, \varepsilon)}{m_q(g)}=1
\end{equation}
where $a_{\text{L}}(q, g, \varepsilon)$ is the number of isogeny classes of $g$-dimensional abelian varieties over $\F_q$ whose largest simple isogeny factor has dimension at least $(1 - \varepsilon) g$ (see Theorem \ref{thm32}).

The main result of this paper is given in Section \ref{Sec4}. Its proof is based on (\ref{eqn2}), (\ref{eqn3}) and some elementary arguments. 

\begin{theorem} \label{thm13}
(Theorem \ref{thm41}) 
\begin{equation}
\limsup_{g \rightarrow \infty}\frac{s_q(g)}{m_q(g)}=1.
\end{equation}
\end{theorem}

The above theorem suggests that there are much more simple isogeny classes of abelian varieties over $\F_q$ of dimension $g$ than non-simple ones for sufficiently large $g$. We expect that the following conjecture to be true. 

\begin{conjecture} \label{conj14}
\begin{equation} 
\lim_{g \rightarrow \infty}\frac{s_q(g)}{m_q(g)}=1. 
\end{equation}
\end{conjecture}


\section{Number of isogeny classes of simple abelian varieties} \label{Sec2}

Let $s_q(g)$ be the number of isogeny classes of $g$-dimensional simple abelian varieties over $\F_q$. Since 
$$
s_q(g) \leq m_q(g) = q^{\frac{1}{4}g^2(1+o(1))}, 
$$
we only need to consider the lower bound of $s_q(g)$. We make critical use of the following Lemma of DiPippo and Howe \cite{DH}, which gives a lower bound for $m_q(g)$ coarsely of the right order of magnitude.

\begin{lemma} \label{lem21}
\textup{(\cite{DH}, Lemma 3.3.1)} Suppose that $a_1, \cdots, a_g$ are integers such that
$$
\left | \frac{a_g}{2q^{g/2}} \right | + \sum_{i=1}^{g-1} \left | \frac{a_i}{q^{i/2}} \right | \leq 1
$$
and $(a_g, q)=1$. Then
$$
F(a_1, \cdots, a_g) := (x^{2g}+q^g)+a_1(x^{2g-1}+q^{g-1}x) + \cdots + a_{g-1}(x^{g+1}+qx^{g-1})+a_gx^g
$$
is a Weil $q$-polynomial. 
\end{lemma}
Let
$$
X_g := \left \{ (a_1, \cdots, a_g) \in \Z^g : \left | \frac{a_g}{2q^{g/2}} \right | + \sum_{i=1}^{g-1} \left | \frac{a_i}{q^{i/2}} \right | \leq 1 \text{ and } (a_g, q)=1 \right \}
$$
(then $\left | X_g \right | = q^{\frac{1}{4}g^2(1+o(1))}$) and 
\begin{equation*} 
\begin{split}
X_g^{\text{sim}} & := \left \{ (a_1, \cdots, a_g) \in X_g : F(a_1, \cdots, a_g) \text{ corresponds to a simple variety} \right \}, \\
X_g^{\text{irr}} & := \left \{ (a_1, \cdots, a_g) \in X_g : F(a_1, \cdots, a_g) \text{ is an irreducible polynomial} \right \} \subset X_g^{\text{sim}}.
\end{split} 
\end{equation*}

We want to prove that 
$$
\left | X_g^{\text{sim}} \right | = q^{\frac{1}{4}g^2(1+o(1))}.
$$ 
One may try to prove this by proving $\left | X_g^{\text{irr}} \right | = q^{\frac{1}{4}g^2(1+o(1))}$, but it is not easy to determine whether given $F(a_1, \cdots, a_g)$ is irreducible or not. For example, the constant term of $F(a_1, \cdots, a_g)$ is $q^g$ so Eisenstein's criterion cannot be applied. Rather than finding a subset of $X_g^{\text{irr}}$ whose size is $q^{\frac{1}{4}g^2(1+o(1))}$, we use different method. 

\begin{lemma} \label{lem22}
Suppose that $m \geq n \geq 1$ are integers and
$$
F(c_1, \cdots, c_{m+n}) = F(a_1, \cdots, a_n) \cdot F(b_1, \cdots, b_m). 
$$
Then $c_{m+1}, \cdots, c_{m+n}$ are determined by $a_1, \cdots, a_n, c_1, \cdots, c_m$. 
\end{lemma}

\begin{proof}
For $1 \leq i \leq m$, $c_i = b_i + g_i(a_1, \cdots, a_n, b_1, \cdots, b_{i-1})$ for some polynomial function $g_i$, which follows by comparing the coefficients of $x^{2(m+n)-i}$ on both sides.
By induction on $i$, one can show that $b_i$ can be represented as a function of $a_1, \cdots, a_n, c_1, \cdots, c_i$. Thus $a_1, \cdots, a_n, c_1, \cdots, c_m$ determine $b_1, \cdots, b_m$ and consequently determine $c_{m+1}, \cdots, c_{m+n}$. 
\end{proof}

\begin{theorem} \label{thm23}
$s_q(g) = q^{\frac{1}{4}g^2(1+o(1))}$. 
\end{theorem}

\begin{proof}
Let
$$
Y_g := \left \{ (a_1, \cdots, a_g) \in \Z^g : \left | \frac{a_g}{2q^{g/2}} \right | \leq \frac{1}{g}, \, \left | \frac{a_i}{q^{i/2}} \right | \leq \frac{1}{g} \,\, (1 \leq i \leq g-1) \text{ and } (a_g, q)=1 \right \} \subset X_g
$$
and
$$
Y_g^{\text{sim}} := Y_g \cap X_g^{\text{sim}} = \left \{ (a_1, \cdots, a_g) \in Y_g : F(a_1, \cdots, a_g) \text{ corresponds to a simple variety} \right \}.
$$
Then clearly $\left | Y_g \right |=q^{\frac{1}{4}g^2(1+o(1))}$. 
Also for any integer $\displaystyle 1 \leq n \leq \frac{g}{2}$, define
\begin{equation*}
Y_{g, n} := \begin{Bmatrix}
(c_1, \cdots, c_g) \in Y_g : F(c_1, \cdots, c_g) = F(a_1, \cdots, a_n) \cdot F(b_1, \cdots, b_{g-n}) \\ 
\text{for some Weil }q\text{-polynomials } F(a_1, \cdots, a_n) \text{ and } F(b_1, \cdots, b_{g-n})
\end{Bmatrix}
\end{equation*}
Then 
$$
Y_g^{\text{sim}} = Y_g \setminus \cup_{1 \leq n \leq \frac{g}{2}} Y_{g, n}
$$
so
\begin{equation}
\left | Y_g^{\text{sim}} \right | \geq \left | Y_g \right |\left ( 1 - \sum_{1 \leq n \leq \frac{g}{2}} \frac{\left | Y_{g, n} \right |}{\left | Y_g \right |} \right ). 
\end{equation}

Now we provide an upper bound of $\displaystyle \frac{\left | Y_{g, n} \right |}{\left | Y_g \right |}$. For any fixed 
$$
\mathbf{c} = (c_1, \cdots, c_{g-n}) \in Z_{g}^{g-n} := \left \{ (a_1, \cdots, a_{g-n}) \in \Z^{g-n} : \left | \frac{a_i}{q^{i/2}} \right | \leq \frac{1}{g} \,\, (1 \leq i \leq g-n) \right \}, 
$$
denote
$$
Y_g(\mathbf{c}) := \left \{ (a_1, \cdots, a_g) \in Y_g : (a_1, \cdots, a_{g-n})=\mathbf{c} \right \}
$$
and
$$
Y_{g, n}(\mathbf{c}) := \left \{ (a_1, \cdots, a_g) \in Y_{g, n} : (a_1, \cdots, a_{g-n})=\mathbf{c} \right \}.
$$
Then 
\begin{center}
$\displaystyle Y_g = \bigsqcup_{\mathbf{c} \in Z_{g}^{g-n}} Y_g(\mathbf{c})$ and 
$\displaystyle Y_{g, n} = \bigsqcup_{\mathbf{c} \in Z_{g}^{g-n}} Y_{g, n}(\mathbf{c})$
\end{center}
so
$$
\frac{\left | Y_{g, n} \right |}{\left | Y_g \right |} \leq \max_{\mathbf{c} \in Z_{g}^{g-n}} \frac{\left | Y_{g, n}(\mathbf{c}) \right |}{\left | Y_g(\mathbf{c}) \right |}.
$$
$\left | Y_g(\mathbf{c}) \right |$ is determined by the choice of $a_{g-n+1}, \cdots, a_g$ and is independent of $\mathbf{c}$. It is given by
\begin{equation*}
\begin{split}
\left | Y_g(\mathbf{c}) \right | & = \prod_{i=g-n+1}^{g-1}\left ( 1+2\left \lfloor \frac{q^{\frac{i}{2}}}{g} \right \rfloor \right ) \cdot \left ( \left ( 1+2\left \lfloor \frac{2q^{\frac{g}{2}}}{g} \right \rfloor \right ) - \left ( 1+2\left \lfloor \frac{2q^{\frac{g}{2}-1}}{g} \right \rfloor \right ) \right ) \\
& > 2(q-1)\left \lfloor \frac{2q^{\frac{g}{2}-1}}{g} \right \rfloor \prod_{i=g-n+1}^{g-1}\frac{q^{\frac{i}{2}}}{g} \\
& > q \left ( \frac{2q^{\frac{g}{2}-1}}{g}-1 \right ) \left ( \frac{q^{\frac{2g-n}{4}}}{g} \right )^{n-1} \\
& = \left ( \frac{2q^{\frac{g}{2}}}{g}-q \right ) \left ( \frac{q^{\frac{2g-n}{4}}}{g} \right )^{n-1}.
\end{split}
\end{equation*}
For large enough $g$, $q^{\frac{g}{2}} > gq$ so 
$$
\left | Y_g(\mathbf{c}) \right | 
> \frac{q^{\frac{g}{2}}}{g} \cdot \left ( \frac{q^{\frac{2g-n}{4}}}{g} \right )^{n-1}
> \left ( \frac{q^{\frac{2g-n}{4}}}{g} \right )^n.
$$
By Lemma \ref{lem22}, an element of $Y_{g, n}(\mathbf{c})$ is determined by a Weil $q$-polynomial $F(a_1, \cdots, a_n)$. By \cite[Corollary 2.3]{LT}, the size of $Y_{g, n}(\mathbf{c})$ is bounded by
$$
\left | Y_{g, n}(\mathbf{c}) \right | 
\leq \prod_{k=1}^{n} (4g \sqrt{q}^k+1) 
< \prod_{k=1}^{n} 5g \sqrt{q}^k 
= (5g q^{\frac{n+1}{4}})^n.
$$
Now we have
$$
\frac{\left | Y_{g, n} \right |}{\left | Y_g \right |} 
\leq \max_{\mathbf{c} \in Z_{g}^{g-n}} \frac{\left | Y_{g, n}(\mathbf{c}) \right |}{\left | Y_g(\mathbf{c}) \right |}
\leq \left ( \frac{5g^2q^{\frac{n+1}{4}}}{q^{\frac{2g-n}{4}}} \right )^n 
\leq \left ( \frac{5g^2}{q^{\frac{g-1}{4}}} \right )^n. 
$$
If $g$ is sufficiently large, then
$$
\left ( \frac{5g^2}{q^{\frac{g-1}{4}}} \right ) \leq \frac{1}{3}
$$
so
$$
\left | Y_g^{\text{sim}} \right | 
\geq \left | Y_g \right | \left ( 1 - \sum_{1 \leq n \leq \frac{g}{2}} \frac{\left | Y_{g, n} \right |}{\left | Y_g \right |} \right ) 
\geq \left | Y_g \right | \left ( 1 - \sum_{1 \leq n \leq \frac{g}{2}} \left ( \frac{1}{3} \right )^n \right ) 
>\frac{\left | Y_g \right |}{2}.
$$
Thus 
$$\left | Y_g^{\text{sim}} \right | = q^{\frac{1}{4}g^2(1+o(1))}$$
and 
$$\left | Y_g^{\text{sim}} \right | \leq \left | X_g^{\text{sim}} \right | 
\leq s_q(g) \leq m_q(g) = q^{\frac{1}{4}g^2(1+o(1))}$$
so 
$$s_q(g) = q^{\frac{1}{4}g^2(1+o(1))}. $$
\end{proof}

By the inequality $\displaystyle \frac{\left | Y_{g, n} \right |}{\left | Y_g \right |} \leq \left ( \frac{5g^2}{q^{\frac{g-1}{4}}} \right )^n$, we have 
$$
\lim_{g \rightarrow \infty}\frac{\left | Y_g^{\text{sim}} \right |}{\left | Y_g \right |} = 1. 
$$
Note that this does not imply
$$
\lim_{g \rightarrow \infty}\frac{s_q(g)}{m_q(g)} = 1, 
$$
which means that the proportion of isogeny classes of $g$-dimensional abelian varieties over $\F_q$ which is simple approaches $1$ as $g \rightarrow \infty$. We provide a weaker version of this statement in Section \ref{Sec4}.


\section{Largest simple isogeny factor} \label{Sec3}

For any $0< \varepsilon < \frac{1}{3}$, let 
\begin{center}
$A_{\text{L}}(q, g, \varepsilon)$ (and $A_{\text{S}}(q, g, \varepsilon)$)
\end{center}
be the set of isogeny classes of $g$-dimensional abelian varieties over $\F_q$ whose largest simple isogeny factor has dimension $\geq (1- \varepsilon)g$ (and $<(1- \varepsilon)g$).

\begin{lemma} \label{lem31}
Every element of $A_{\text{S}}(q, g, \varepsilon)$ can be represented by $B_1 \times B_2$ for some abelian varieties $B_1$ and $B_2$ over $\F_q$ such that 
$\varepsilon g \leq \text{dim}B_1, \text{dim}B_2 \leq (1 - \varepsilon)g$. 
\end{lemma}

\begin{proof}
Let $A \in A_{\text{S}}(q, g, \varepsilon)$ and suppose that $A$ is isogenous to $A_1 \times \cdots \times A_r$ where $A_1, \cdots, A_r$ are simple, $\text{dim}A_i=d_i$ with $d_1 \geq \cdots \geq d_r>0$. If $d_1 \geq \varepsilon g$, Then $B_1=A_1$ and $B_2=A_2 \times \cdots \times A_r$ satisfy the condition. If $d_1 < \varepsilon g$, then the length of an interval $(\varepsilon g, (1 - \varepsilon) g)$ is larger than each of $d_i$ ($1 \leq i \leq r$) so there exists a positive integer $1<k<r$ such that 
$$
\varepsilon g < \sum_{i=1}^{k}d_i < (1 - \varepsilon) g.
$$
In this case $B_1 = A_1 \times \cdots \times A_k$ and $B_2 = A_{k+1} \times \cdots \times A_r$ satisfy the statement of the lemma. 
\end{proof}

Define
\begin{center}
$a_{\text{L}}(q, g, \varepsilon) := \left | A_{\text{L}}(q, g, \varepsilon) \right |$ and 
$a_{\text{S}}(q, g, \varepsilon) := \left | A_{\text{S}}(q, g, \varepsilon) \right |$.
\end{center}
Then we have the following result.

\begin{theorem} \label{thm32}
\textup{(a)} $\log a_{\text{S}}(q, g, \varepsilon) < \frac{1}{4}(1-\varepsilon)g^2 \log q$ for sufficiently large $g$. \\
\textup{(b)} $\displaystyle \lim_{g \rightarrow \infty} \frac{a_{\text{L}}(q, g, \varepsilon)}{m_q(g)}=1$ and $\displaystyle \lim_{g \rightarrow \infty} \frac{a_{\text{S}}(q, g, \varepsilon)}{m_q(g)}=0$.
\end{theorem}

\begin{proof}
(a) By Lemma \ref{lem31}, every element of $A_{\text{S}}(q, g, \varepsilon)$ can be represented by $B_1 \times B_2$ such that $\varepsilon g \leq \text{dim}B_1, \text{dim}B_2 \leq (1 - \varepsilon)g$ so 
$$
a_{\text{S}}(q, g, \varepsilon) \leq \sum_{\varepsilon g \leq i \leq (1 - \varepsilon)g}m_q(i)m_q(g-i).
$$
For $\displaystyle \varepsilon' := \frac{1-\frac{4}{3}\varepsilon}{(1-\varepsilon)^2+\varepsilon^2}-1>0$, 
$m_q(g) = q^{\frac{1}{4}g^2(1+o(1))}$ so there is $g_0$ such that for every $g \geq g_0$, 
$$
g \cdot q^{\frac{1}{4}g^2(1-\varepsilon)} < m_q(g) < q^{\frac{1}{4}g^2(1+\varepsilon')}. 
$$
Let $\displaystyle g \geq \frac{g_0}{\varepsilon}$. Then 
\begin{equation*}
\begin{split}
a_{\text{S}}(q, g, \varepsilon) & \leq \sum_{\varepsilon g \leq i \leq (1 - \varepsilon)g}q^{\frac{1}{4}i^2(1+\varepsilon')}q^{\frac{1}{4}(g-i)^2(1+\varepsilon')} \\
& \leq g \cdot q^{\frac{1}{4}((\varepsilon )^2 + (1-\varepsilon)^2)g^2(1+\varepsilon')} \\
& = g \cdot q^{\frac{1}{4}(1 - \frac{4}{3} \varepsilon)g^2}. 
\end{split}
\end{equation*}
Thus $\log a_{\text{S}}(q, g, \varepsilon) \leq \log g + \frac{1}{4}(1-\frac{4}{3}\varepsilon) g^2 \log q < \frac{1}{4}(1-\varepsilon)g^2 \log q$ for sufficiently large $g$. \\

\noindent (b) For $\displaystyle g \geq \frac{g_0}{\varepsilon}$, 
$$
a_{\text{S}}(q, g, \varepsilon) \leq g \cdot q^{\frac{1}{4}(1 - \frac{4}{3} \varepsilon)g^2} < q^{- \frac{1}{12} \varepsilon g^2} m_q(g)
$$
so $\displaystyle \lim_{g \rightarrow \infty} \frac{a_{\text{S}}(q, g, \varepsilon)}{m_q(g)}=0$ and
$\displaystyle \lim_{g \rightarrow \infty} \frac{a_{\text{L}}(q, g, \varepsilon)}{m_q(g)}=1$.
\end{proof} 


\section{Proportion of simple isogeny classes} \label{Sec4}

In this section we prove the following theorem based on the results of the previous sections. This is the main theorem of our paper. 

\begin{theorem} \label{thm41}
$$
\limsup_{g \rightarrow \infty}\frac{s_q(g)}{m_q(g)}=1.
$$
\end{theorem}

\begin{proof}
Let $C_1 \geq 1$ be a constant such that $m_q(g) \leq q^{C_1 g^2}$ for any $g \geq 1$, and let $\displaystyle \varepsilon := \frac{1}{16C_1}>0$. Suppose that there exists $c<1$ such that $s_q(g) \leq c m_q(g)$ for sufficiently large $g$. 
Since
$$
\lim_{g \rightarrow \infty} \frac{a_{\text{L}}(q, g, \varepsilon)}{m_q(g)}=1
$$
by Theorem \ref{thm32} and
$$
a_{\text{L}}(q, g, \varepsilon) = s_q(g)+ \sum_{1 \leq i \leq \varepsilon g} m_q(i)s_q(g-i),
$$
there exist $C_2 \geq 1$ and $g_1>0$ such that 
\begin{equation} \label{eqnA1}
s_q(g) \leq C_2 \sum_{1 \leq i \leq \varepsilon g} m_q(i)s_q(g-i) \leq C_2 \sum_{1 \leq i \leq \varepsilon g} q^{C_1 i^2}s_q(g-i)
\end{equation}
for any $g \geq g_1$. In (\ref{eqnA1}), $s_q(g)$ is bounded by a linear combination of 
$$
s_q(g-1), \cdots, s_q(g-n)
$$ 
where $n := \left \lfloor \varepsilon g \right \rfloor$.

Now we prove that the asymptotic of $s_q(g)$ in Theorem \ref{thm23},
$$
s_q(g) = q^{\frac{1}{4}g^2(1+o(1))}
$$
makes a contradiction. If we replace each $s_q(g-i)$ in (\ref{eqnA1}) by $q^{\frac{1}{4}(g-i)^2}$, then
\begin{equation} \label{eqnA2}
C_2 \sum_{1 \leq i \leq \varepsilon g} q^{C_1 i^2} q^{\frac{1}{4}(g-i)^2}
\leq C_2  \sum_{1 \leq i \leq \varepsilon g} q^{\frac{1}{16}gi} q^{\frac{1}{4}(g-i)^2}
\leq C_2 \varepsilon g \cdot q^{\frac{1}{4}(g-1)^2+\frac{1}{16}g}
\end{equation}
and the right side of (\ref{eqnA2}) becomes smaller than $q^{\frac{1}{4}g^2}$ for sufficiently large $g$. This suggests that it would be possible to obtain a contradiction by iteratively bounding the numbers $s_q(g), s_q(g+1), \cdots $ by linear combinations of $s_q(g-1), s_q(g-2), \cdots, s_q(g-n)$. (For example, $s_q(g+1)$ is bounded by a linear combination of $s_q(g), \cdots, s_q(g-n+1)$ and $s_q(g)$ is bounded by a linear combination of $s_q(g-1), \cdots, s_q(g-n)$ so $s_q(g+1)$ is also bounded by a linear combination of $s_q(g-1), \cdots, s_q(g-n)$.)

For
$$
\varepsilon' := \frac{\varepsilon^2}{4+(2+\varepsilon)^2}>0,
$$
there exists $g_2>0$ such that for any $g \geq g_2$, 
$$
q^{\frac{1}{4}g^2(1-\varepsilon')} \leq s_q(g) \leq q^{\frac{1}{4}g^2(1+\varepsilon')}
$$
by Theorem \ref{thm23}. Denote 
$$a_{g,k} := s_q(g+k)$$
and 
$$b_{g,k} := \sum_{i=1}^{n} q^{C_1 i^2}s_q(g+k-i).$$ 
Then 
$$
a_{g,k} \leq C_2 b_{g,k}
$$
for $0 \leq k \leq n$ by (\ref{eqnA1}). 
Now suppose that  $g \geq \max\left \{ g_1, 2g_2, 2 \varepsilon^{-1} \right \}$ (so $n \geq 2$) and consider the inequality
\begin{equation} \label{eqn6}
a_{g,n}+\sum_{k=1}^{n} 2^{k-1}C_2^{k}q^{C_1 k^2} a_{g,n-k} \leq C_2 b_{g,n} +\sum_{k=1}^{n} 2^{k-1}C_2^{k+1}q^{C_1 k^2} b_{g,n-k}. 
\end{equation}
For $1 \leq k \leq n$, the coefficient of $s_q(g+n-k)$ in the left side of (\ref{eqn6}) is 
$$
2^{k-1}C_2^k q^{C_1k^2}
$$
and the coefficient of $s_q(g+n-k)$ in the right side of (\ref{eqn6}) is 
\begin{equation*}  \label{eqn8}
\begin{split}
& C_2 q^{C_1 k^2}+\sum_{j=1}^{k-1}(2^{j-1}C_2^{j+1}q^{C_1j^2})q^{C_1(k-j)^2} \\
\leq & \left ( 1+\sum_{j=1}^{k-1}2^{j-1} \right ) C_2^k q^{C_1k^2}=2^{k-1}C_2^k q^{C_1k^2}.
\end{split}
\end{equation*}
Now (\ref{eqn6}) provides an upper bound of $s_q(g+n)=a_{g,n}$. It is given by
\begin{equation} \label{eqn7}
\begin{split}
s_q(g+n) & \leq \sum_{k=1}^{n}\sum_{i=n-k+1}^{n}2^{k-1}C_2^{k+1}q^{C_1k^2}q^{C_1 i^2}s_q(g+n-k-i) \\
& \leq \frac{n(n+1)}{2} 2^{n-1}C_2^{n+1}q^{2C_1 n^2} \max\left \{ s_q(t) \mid g-n \leq t \leq g-1  \right \} \\
& \leq \frac{\varepsilon g(\varepsilon g+1)}{2} 2^{\varepsilon g-1}C_2^{\varepsilon g+1}q^{\frac{1}{8} \varepsilon g^2} q^{\frac{1}{4}g^2(1+\varepsilon')}
\end{split}
\end{equation}
(the last inequality is due to the fact that $\displaystyle g-n>\frac{g}{2} \geq g_2$). By the definition of $\varepsilon'$, 
$$
\frac{1}{4}(g+n)^2(1-\varepsilon')>\frac{1}{4}(g+\frac{\varepsilon g}{2})^2(1-\varepsilon')=\frac{1}{4}g^2(1+\varepsilon')+\frac{1}{4}\varepsilon g^2
$$
so
\begin{equation} \label{eqn8}
s_q(g+n) \geq q^{\frac{1}{4}(g+n)^2(1-\varepsilon')}
>q^{\frac{1}{4}g^2(1+\varepsilon')+\frac{1}{4}\varepsilon g^2}.
\end{equation}
By (\ref{eqn7}) and (\ref{eqn8}), we obtain
$$
q^{\frac{1}{4}\varepsilon g^2} < \frac{\varepsilon g(\varepsilon g+1)}{2} 2^{\varepsilon g-1}C_2^{\varepsilon g+1}q^{\frac{1}{8}\varepsilon g^2}, 
$$
which is a contradiction when $g$ is large enough. 
\end{proof}

\vspace{2mm}


\section*{Acknowledgments}

This work was partially supported by Samsung Science and Technology Foundation (SSTF-BA1802-03) and National Research Foundation of Korea (NRF-2018R1A4A1023590). The author would like to thank Sungmun Cho, Dong Uk Lee, Donghoon Park and Jacob Tsimerman for their interest and helpful comments. The author also deeply thank the anonymous referee for their comments that improved the exposition of the paper. 


\vspace{3mm}

\footnotesize{
\textsc{Jungin Lee: Department of Mathematics, Pohang University of Science and Technology, 77 Cheongam-ro, Nam-gu, Pohang, Gyeongbuk, Republic of Korea 37673.} 

\textit{E-mail address}: \changeurlcolor{black}\href{mailto:moleculesum@postech.ac.kr}{moleculesum@postech.ac.kr} 

\end{document}